\documentclass[10pt]{amsart}
\usepackage{amsmath}
\usepackage{amssymb}
\usepackage{amsfonts}
\usepackage{amsthm}
\usepackage{enumerate}
\usepackage[all]{xy}
\usepackage[latin1]{inputenc}
\usepackage{mathdots}
\usepackage{graphicx}
\usepackage{latexsym}
\usepackage{mathrsfs}

\usepackage{tikz}

\PassOptionsToPackage{usenames,dvipsnames,svgnames}{xcolor}

\input xy

\newtheorem{Theo}{Theorem}[section]

\newtheorem{Prop}[Theo]{Proposition}

\newtheorem{Cor}[Theo]{Corollary}
\newtheorem{Lemma}[Theo]{Lemma}

\theoremstyle{definition}

\newtheorem{Exam}[Theo]{Example}
\newtheorem{Remark}[Theo]{Remark}

\def\mystrut(#1,#2){\vrule height #1pt depth #2pt width 0pt}

\newcommand{\Hom}{{\rm Hom}}
\newcommand{\Ext}{{\rm Ext}}

\newcommand{\mmod}{{\rm mod}\,}

\newcommand{\Z}{\mathbb{Z}}

\begin{document}

\title[Idempotent reduction for the finitistic dimension conjecture]{Idempotent reduction for the finitistic dimension conjecture}

\author{Diego Bravo}
\address{Instituto de Matem\'atica y Estad\'istica "Rafael Laguardia", Universidad de la Rep\'ublica,
Julio Herrera y Reissig 565, Montevideo, Uruguay}
\email{dbravo@fing.edu.uy}
\author{Charles Paquette}\address{Department of Mathematics and Computer Science, Royal Military College of Canada,
Kingston, ON K7K 7B4, Canada}
\email{charles.paquette.math@gmail.com}


\thanks{ The second author was supported by the Natural Sciences and Engineering Research Council of Canada and by CDARP. Both authors would like to thank an anonymous referee for his comments, which led to improved upper bounds.}

\begin{abstract}In this note, we prove that if $\Lambda$ is an Artin algebra with a simple module $S$ of finite projective dimension, then the finiteness of the finitistic dimension of $\Lambda$ implies that of $(1-e)\Lambda(1-e)$ where $e$ is the primitive idempotent supporting $S$. We derive some consequences of this. In particular, we recover a result of Green-Solberg-Psaroudakis: if $\Lambda$ is the quotient of a path algebra by an admissible ideal $I$ whose defining relations do not involve a certain arrow $\alpha$, then the finitistic dimension of $\Lambda$ is finite if and only if the finitistic dimension of $\Lambda/\Lambda\alpha \Lambda$ is finite.
\end{abstract}

\maketitle

\section{Introduction}

In this paper, $\Lambda$ stands for an Artin algebra, that we will assume to be basic. Since we are mainly interested in homological properties of the finitely generated modules over $\Lambda$, this is not a restriction. We denote by $\mmod \Lambda$ the category of finitely generated left $\Lambda$-module. All modules considered are left modules, unless otherwise stated. We let $e$ denote an idempotent of $\Lambda$ and $S_e$ be the semi-simple module supported at $e$. If ${\rm rad} \Lambda$ is the Jacobson radical of $\Lambda$, then $S_e \cong \Lambda e/{\rm rad} \Lambda e$. Note that $S_e$ is simple if and only if $e$ is primitive. We will consider the Artin algebra $(1-e)\Lambda(1-e)$, which will be denoted by $\Gamma$. In other words, $\Gamma$ is the endomorphism algebra over $\Lambda$ of the projective module $\Lambda (1-e)$.

It is desirable to relate the homological properties of $\Lambda$ and $\Gamma$. This has been achieved in \cite{APT, IP, IP2}, for instance. However, one has to impose some conditions on $e$ as in general, the Artin algebras $\Lambda$ and $\Gamma$ can be very far apart, from a homological perspective.  In \cite{IP}, it was shown that when $e$ is primitive and $\Lambda$ is finite dimensional over an algebraically closed field, the finiteness of the global dimensions of $\Lambda$ and $\Gamma$ are equivalent, when all higher self-extension groups of $S_e$ vanish. Moreover, the latter condition happens to be necessary for both $\Lambda$ and $\Gamma$ having finite global dimension.

We let
$${\rm findim}\,\Lambda = {\rm sup\,}\{{\rm pd}_\Lambda M \mid M \in {\rm mod}\Lambda, \, {\rm pd}_\Lambda M < \infty\}$$
denote the (little) \emph{finitistic dimension} of $\Lambda$, that is, the {\sl supremum} of the projective dimensions of those finitely generated $\Lambda$-modules having finite projective dimension. It has been conjectured (and publicized by Bass) in the sixties, see \cite{Ba}, that the finitistic dimension of an Artin algebra is always finite. This is a very important problem in representation theory of algebras, as a positive answer to it implies the validity of many other important homological conjectures, including the Gorenstein symmetry conjecture, the Nakayama conjecture, the generalized Nakayama conjecture, Nunke's condition, the Auslander-Reiten conjecture and the vanishing conjecture; see \cite{Happel}. In this note, we are mainly interested in comparing the finitistic dimensions of closely related algebras such as $\Lambda$ and $\Gamma$. It is worth noting that comparing the finitistic dimensions of such algebras is not new. It has been done, for instance, by Fuller and Saorin in \cite{FS}, and by Xi in \cite{CC}. In this paper, we will give a particular attention to the case where $e$ is primitive. Our first main result is the following.

\begin{Theo} Let $e$ be a primitive idempotent such that ${\rm pd}_\Lambda S_e$ is finite. Then ${\rm findim}\, \Gamma \le 2\,{\rm findim}\,\Lambda - \ell$ for some $1 \le \ell \le {\rm pd}_\Lambda S_e$.
\end{Theo}

As we will see later, the condition that $e$ is primitive can be slightly relaxed. Note also that the condition that ${\rm pd}_\Lambda S_e < \infty$ is not that restrictive, due to the following result by Auslander \cite{Aus}, later generalized by Iyama \cite{Iyama}.

\begin{Prop}[Auslander] Given an Artin algebra $\Gamma$, there exists an Artin algebra $\Lambda$ with ${\rm gl.dim}\,\Lambda < \infty$ and $\Gamma = (1-e)\Lambda(1-e)$ for some idempotent $e$ of $\Lambda$.
\end{Prop}

In particular, in this proposition, we have ${\rm pd}_\Lambda S_e < \infty$ and ${\rm findim}(\Lambda)<\infty$. Therefore, if we could extend the above theorem to $e$ not necessarily primitive, that would yield a complete proof of the finitistic dimension conjecture. In the second part of the note, we provide applications of the above theorem.

Our first application is the following result, that was first proven in \cite{GSP} when $x$ is an arrow. Here, $k$ is a field, $Q = (Q_0, Q_1)$ a finite quiver and $I$ an admissible ideal of the path algebra $kQ$.

\begin{Prop}
Let $\Lambda = kQ/I$ be such that $I$ is generated by relations not involving a given element $x$ of $Q_0 \cup Q_1$. Let $J$ be the two sided ideal of $\Lambda$ generated by $x$. Then ${\rm findim}\,\Lambda \le 2\,{\rm findim}\,\Lambda/J +3 \le 2\,{\rm findim}\,\Lambda +3$.
\end{Prop}

Our second application is the following, where $e$ is a primitive idempotent; compare \cite[Corollary 3.2]{Wei}.

\begin{Prop}Let $\Lambda = kQ/I$ as above. Let $J$ be a submodule of $\Lambda e$ with $J {\rm rad} \Lambda=0$. Thus, $J$ is a two-sided ideal.
Assume further that ${\rm pd}_{\Lambda/J}J < \infty$. Then ${\rm findim}\,\Lambda \le 2\,{\rm findim}\,\Lambda/J + 3 $.
\end{Prop}

\section{Main result}

We let $F:= \Hom_\Lambda(\Lambda (1-e),-)$ be the exact functor from $\mmod \Lambda$ to $\mmod \Gamma$. Observe that $F$ can be extended naturally to an exact functor from the corresponding bounded derived categories $D^b(\mmod \Lambda) \cong K^{-,b}({\rm proj} \Lambda)$ and $D^b(\mmod \Gamma) \cong K^{-,b}({\rm proj} \Gamma)$, by taking the right derived functor $R\Hom_\Lambda(\Lambda(1-e),-)$, that will also be denoted by $F$. Now, consider the functor $G:= \Lambda(1-e) \otimes_\Gamma -: \mmod \Gamma \to \mmod \Lambda$. Similarly, $G$ can be extended naturally to an exact functor from the corresponding derived categories $D(\mmod \Gamma)$ and $D(\mmod \Lambda)$, by taking the left derived functor $\Lambda(1-e) \otimes_\Gamma^L -$, that will also be denoted by $G$. Notice that unbounded derived categories need to be used.
Note also that as a functor between module categories, $G$ is not exact. It is well known that $(G,F)$ is an adjoint pair, both at the module category level and at the derived category level.

\begin{Lemma} The co-unit $\eta: 1 \to FG$ of the adjunction at the module category level is a natural isomorphism.
\end{Lemma}

\begin{proof}
Observe that $\Hom_\Lambda(\Lambda(1-e),\Lambda(1-e)\otimes_\Gamma X)$ is naturally isomorphic to $(1-e)(\Lambda(1-e)\otimes_\Gamma X) = \Gamma \otimes_\Gamma X$, which is naturally isomorphic to $X$. This series of isomorphisms sends $f \in \Hom_\Lambda(\Lambda(1-e),\Lambda(1-e)\otimes_\Gamma X)$ to $f(1-e) \in \Gamma \otimes_\Gamma X$, where the latter can be written uniquely as a simple tensor $(1-e) \otimes x$ for some $x \in X$. On the other hand, we have $\eta_X: X \to \Hom_\Lambda(\Lambda(1-e),\Lambda(1-e)\otimes_\Gamma X)$, where for $x \in X$ and $a \in \Lambda$, we have $\eta_X(x)(a(1-e)) = a(1-e) \otimes x$. Hence, $\eta_X$ is an isomorphism.
\end{proof}

This implies that $G$ is fully faithful. Let $\mathcal{B}$ be the full subcategory of $\mmod \Lambda$ generated by the $GX$ for $X$ in $\mmod \Gamma$. The above implies that $\mmod \Gamma$ is equivalent to $\mathcal{B}$ through $G$ and $F$. Finally, observe that for an idempotent $e'$ of $\Lambda$ with $e'(1-e)=e'$, we have $G(\Gamma e') \cong \Lambda e'$ and $F(\Lambda e') \cong \Gamma e'$. Therefore, $G$ sends projective $\Gamma$-modules to projective $\Lambda$-modules while $F$ sends projective modules in add$(\Lambda(1-e))$ to projective $\Gamma$-modules. Observe that the objects in add$(\Lambda(1-e))$ are precisely the projective objects in $\mathcal{B}$.

\medskip

Let $Y(e)$ denote the Ext-algebra of $S_e$. It is a positively graded algebra whose degree $i$ piece is $\Ext_\Lambda^i(S_e, S_e)$. It is an Artin algebra when, for instance, ${\rm pd}_\Lambda S_e < \infty$. We say that a graded Artin algebra $\Lambda$ has \emph{uniform graded (left) loewy length} if all the indecomposable (graded) projective modules in $\mmod \Lambda$ have the same graded loewy length. Here, the \emph{graded loewy length} of a graded module $M = (M_i)_{i \in \Z}$ is the difference $r-s$ where $r$ is maximal with $M_r \ne 0$ and $s$ is minimal with $M_s \ne 0$. The reason why this concept of uniform graded loewy length is important lies in the following.

\begin{Lemma} Let $Y(e)$ be an Artin algebra of uniform graded loewy length $\ell$. Let $M$ be a left $\Lambda/\Lambda(1-e)\Lambda$-module. Then, as $\Lambda$-modules, we have $\Ext_\Lambda^{\ell-1}(M, S_e) \ne 0$ and $\Ext_\Lambda^{i}(M, S_e) = 0$ for $i \ge \ell$.
\end{Lemma}

\begin{proof} The condition implies that for any indecomposable direct summand $S$ of $S_e$, we have $\Ext_\Lambda^{\ell-1}(S, S_e) \ne 0$ and $\Ext_\Lambda^{i}(S, S_e) = 0$ for $i \ge \ell$. Observe that as a $\Lambda$-module, any $\Lambda/\Lambda(1-e)\Lambda$-module has all of its composition factors in ${\rm add}S_e$, the full additive subcategory of $\mmod \Lambda$ generated by the direct summands of $S_e$. Therefore, we have a short exact sequence
$$0 \to S \to M \to M' \to 0$$
where $S \in {\rm add}S_e$ and $M'$ is an $\Lambda/\Lambda(1-e)\Lambda$-module. The rest of the proof goes by induction on the length of $M$ by considering the long exact sequence obtained after applying the functor $\Hom_\Lambda(-,S_e)$ to the above short exact sequence.
\end{proof}

The following lemma is crucial for obtaining our main theorem.

\begin{Lemma}\label{nocohdegi}
Assume that ${\rm findim}(\Lambda)=r < \infty$, that ${\rm pd}_\Lambda S_e < \infty$ and that $Y(e)$ has uniform graded loewy length $\ell$. Let $N$ be a $\Gamma$-module of finite projective dimension $s$. Then the minimal projective resolution of the complex $GN$ has no cohomology in degree $i$ where $i \le {\rm min}(-1,-r+\ell-1)$.
\end{Lemma}

\begin{proof}
Consider a minimal projective resolution
$$0 \to P_{-s} \to P_{-s+1} \to \cdots \to P_{-1} \to P_0$$
of $N$. Then
$$C:=G(P_{-s}) \to G(P_{-s+1}) \to \cdots \to G(P_{-1}) \to G(P_0)$$
is a minimal projective resolution of $GN$. Note that $C$ need not be a module. Note also that each cohomology of $C$ occurring in negative degree has the structure of an $\Lambda/\Lambda(1-e)\Lambda$-module. Let $-s \le t_1 \le 0$ be the least degree for which the above complex has non-zero cohomology $Z_1$. If $t_1=0$, then there is nothing to prove. So assume $t_1 <0$. Note that as a $\Lambda$-module, $Z_1$ has all of its composition factors in ${\rm add}S_e$. As a consequence, using that ${\rm pd}_\Lambda S_e < \infty$, we get that $Z_1$ has finite projective dimension as a $\Lambda$-module. Moreover, since $Y(e)$ has uniform graded loewy length $\ell$, it follows from the above lemma that $\Ext^i_\Lambda(Z_1,S_e)=0$ for all $i \ge \ell$ and $\Ext^{\ell - 1}_\Lambda(Z_1,S_e)\ne 0$. By minimality of $t_1$, we have a canonical morphism $Z_1[-t_1] \to C$, which induces an isomorphism in degree $t_1$ cohomologies. There is a corresponding exact triangle
$$Z_1[-t_1] \to C \to C_1$$
and it follows that $C_1$ has a finite (possibly non-minimal) projective resolution of shape
$$\to G(P_{t_1-2})\oplus Q_{-1} \to G(P_{t_1-1})\oplus Q_0 \to G(P_{t_1}) \to G(P_{t_1+1}) \to \cdots \to G(P_{-1}) \to G(P_0)$$
$$\cdots \to Q_{-t_1-s} \to G(P_{-s})\oplus Q_{-t_1-s+1} \to G(P_{-s+1})\oplus Q_{-t_1-s+2} \to \cdots$$
where $\cdots \to Q_{-1} \to Q_0$ is the minimal projective resolution of $Z_1$, and where we have set $Q_i=0$ for $-i > {\rm pd}_\Lambda Z_1$. Also, $C_1$ is now exact in degrees $\le t_1$. We continue inductively: suppose that $C_i$ has been constructed and let $t_{i+1} < 0$ be the least integer, if any, such that $C_i$ has non-zero cohomology $Z_{i+1}$ in degree $t_{i+1}$. Then $\Ext^j_\Lambda(Z_{i+1},S_e)=0$ for all $j \ge \ell$ and $\Ext^{\ell - 1}_\Lambda(Z_{i+1},S_e)\ne 0$. Consider the canonical morphism $Z_{i+1}[-t_{i+1}] \to C_i$ with exact triangle
$$Z_{i+1}[-t_{i+1}] \to C_i \to C_{i+1}$$
where $C_{i+1}$ is exact in degrees $\le t_{i+1}$. Now, there is some $m \ge 1$ such that $C_m$ is quasi-isomorphic to a module with finite projective dimension. Assume to the contrary that $GN \cong C$ has cohomology in degree $j$ where $j \le {\rm min}(-1,-r+\ell-1)$. That means $t_1 \le j \le {\rm min}(-1,-r+\ell-1)$. So $t_1$ is negative and $t_1 \le -r + \ell-1$. By applying the functor $\Hom(-, S_e)$ and using that for all $i \ge 1$, we have $\Hom(Z_i, S_e[\ell-1]) \ne 0$ and that $\Hom(Z_i, S_e[q])=0$ for all $q \ge \ell$, we get $$\Hom(C_1, S_e[-t_1 + \ell]) \cong \Hom(Z_1[-t_1], S_e[-t_1 + \ell-1]) \ne 0$$ and $$\Hom(C_1, S_e[-t_1 + q+1]) \cong \Hom(Z_1[-t_1], S_e[-t_1 + q]) = 0$$ for all $q \ge \ell$. By induction, we get $$\Hom(C_i, S_e[-t_1 + \ell]) \cong \Hom(C_{i+1}, S_e[-t_1 + \ell]) \ne 0$$ and $$\Hom(C_{i+1}, S_e[-t_1 + q+1]) =0$$ for all $q \ge \ell$. In particular, we get $$\Hom(C_{m}, S_e[-t_1 + \ell]) \ne 0,$$ so $${\rm pd}_\Lambda C_m \ge -t_1 + \ell \ge r -\ell + 1 + \ell = r +1,$$ a contradiction to ${\rm findim}(\Lambda)=r$.
\end{proof}

\begin{Cor} \label{Cor}
Assume that ${\rm findim}\,\Lambda=r < \infty$, that ${\rm pd}_\Lambda S_e < \infty$ and that $Y(e)$ has uniform graded loewy length $\ell$. Then ${\rm findim}\,\Gamma \le 2r-\ell$.
\end{Cor}

\begin{proof} Let $N$ be a $\Gamma$-module with projective dimension $s$. By lemma \ref{nocohdegi}, the minimal projective resolution of $GN$ has no cohomology in degree $i$ where $i \le {\rm min}(-1,-r+\ell-1)$. Therefore, we get a $\Lambda$ module with projective dimension equal to $s-(r - \ell +1)+1$ whenever $s-r+\ell$ is non-negative. Since ${\rm findim}\,\Lambda=r$, we get $s-r+\ell \le r$ when $s-r+\ell \ge 0$, so $s \le 2r - \ell$. If $s-r+\ell < 0$, then $s \le r - \ell$ so $s \le 2r - \ell$ as well.
\end{proof}

We are ready for our main theorem.

\begin{Theo} \label{maintheo}
Assume that ${\rm findim}\,\Lambda=r < \infty$ and ${\rm pd}_\Lambda S_e < \infty$ where $e$ is primitive. Then ${\rm findim}\,\Gamma \le 2r-\ell$ where $\ell < \infty$ is the greatest positive integer with $\Ext^{\ell-1}_\Lambda(S_e, S_e) \ne 0$.
\end{Theo}

\begin{proof}
Since $e$ is primitive, the algebra $Y(e)$ only has one indecomposable projective module, which is of finite length since ${\rm pd}_\Lambda S_e < \infty$. Therefore, the algebra $Y(e)$ has uniform graded loewy length, which is the greatest positive integer $\ell$ with $\Ext^{\ell-1}_\Lambda(S_e, S_e) \ne 0$. The result follows from Corollary \ref{Cor}.
\end{proof}

\begin{Remark} Let $\Lambda = kQ/I$ where $Q$ is the quiver
$$\xymatrixrowsep{10pt}\xymatrixcolsep{10pt}\xymatrix{
 & 1 \ar[dr]^{\alpha_1} & \\
3 \ar[ur]^{\alpha_3} && 2 \ar[ll]^{\alpha_2}\\
}$$
We let $I = \langle \alpha_3\alpha_2\alpha_1 \rangle$. It is easy to check that ${\rm findim}(\Lambda) = {\rm gl.dim}(\Lambda)=2$. Take $e = e_1$. Observe that $\ell = 3$, as $\Ext_\Lambda^2(S_1,S_1)\ne 0$ but ${\rm pd}_\Lambda S_1 = 2$. The theorem states that ${\rm findim}(1-e_1)\Lambda(1-e_1) \le 2\cdot 2 - 3 = 1$. An easy check shows that this bound is attained.
\end{Remark}

\section{Applications}

In this section, we assume that $\Lambda$ is finite-dimensional over an algebraically closed field $k$. For our applications, there is no loss of generality in assuming that $\Lambda$ is basic. Therefore, there exists a finite quiver $Q$ and an admissible ideal $I$ of $kQ$ such that $\Lambda$ is isomorphic to $kQ/I$. We will therefore assume that $\Lambda = kQ/I$.

\subsection{Projective ideals}

We start with the following result, whose proof is essentially due to \'Agoston, Happel, Luk\'acs and Unger; see \cite{AHLU}.

\begin{Prop} \label{AHLU} Let $e$ be primitive such that $J:=\Lambda e \Lambda$ is projective as a left module. Then
\begin{enumerate}[$(1)$]
    \item ${\rm findim}\Lambda \le {\rm findim}\Lambda/J +2$.
\item ${\rm findim}\Lambda/J \le {\rm findim}\Lambda.$
\item ${\rm findim}\Lambda < \infty$ if and only if ${\rm findim}\Lambda/J < \infty$.\end{enumerate}
\end{Prop}

\begin{proof} The last part is a consequence of the first two parts. For the first part, see \cite[Theorem 2.2]{AHLU}. Let us prove the second part. Assume that ${\rm findim}\Lambda$ is finite and equal to $r$. Let $M$ be a finitely generated left $\Lambda/J$-module (that is, a finitely generated left $\Lambda$-module with $JM=0$) of finite projective dimension. Observe that the indecomposable projective $\Lambda/J$-modules are $\Lambda e_j/J e_j$, for $1 \le j \le n$. Since $J e_j$ is projective, we get that a projective $\Lambda/J$-module has projective dimension at most one, when seen as a $\Lambda$-module. This means that $M$ has finite projective dimension as a $\Lambda$-module, so ${\rm pd}_\Lambda M \le r$, which implies ${\rm pd}_{\Lambda/J} M \le r$ by \cite[Lemma 1.2]{LP}.
\end{proof}

\begin{Remark}
Note that if $e$ is a general idempotent with $\Lambda e \Lambda$ projective as a left module, then $\Lambda e \Lambda$ is a stratifying ideal. It is well known that in this case, there is a recollement
$$\xymatrix{D^-(\Lambda/\Lambda e \Lambda) \ar[r] & D^-(\Lambda) \ar@<8pt>[l] \ar@<-8pt>[l] \ar[r] & D^-(e\Lambda e) \ar@<8pt>[l] \ar@<-8pt>[l]}$$
of the derived categories. Note that the fact that $\Lambda e \Lambda$ is projective as a left module implies that $e\Lambda$ is a projective $e \Lambda e$-module. If we assume further that $\Lambda e$ has finite projective dimension as a right $e \Lambda e$-module, then the above recollement restricts to a recollement
$$\xymatrix{D^b(\Lambda/\Lambda e \Lambda) \ar[r] & D^b(\Lambda) \ar@<8pt>[l] \ar@<-8pt>[l] \ar[r] & D^b(e\Lambda e) \ar@<8pt>[l] \ar@<-8pt>[l]}$$
and in this case Happel has shown in \cite{Happel} that ${\rm findim} \,\Lambda < \infty$ if and only if ${\rm findim} \,\Lambda/\Lambda e \Lambda < \infty$ and ${\rm findim}\, e\Lambda e < \infty$. Note that when $e$ is primitive, one has ${\rm findim} \,e\Lambda e = 0$. However, the condition that $\Lambda e$ has finite projective dimension as a right $e \Lambda e$-module is not automatically satisfied. Therefore, when $e$ is primitive, Proposition \ref{AHLU} is stronger than the above fact.
\end{Remark}

An element $r \in kQ$ is called \emph{uniform} if it is a linear combination of parallel paths. Let us fix a finite set $\{r_1, r_2, \ldots, r_t\}$ of uniform generators of $I$. If $x$ is an element of $Q_0 \cup Q_1$ (which can naturally be thought of as an element of $kQ$ or of $\Lambda$) and $r \in kQ$ is such that no term of $r$ can be factorized through $x$, then we say that $r$ \emph{does not involve} $x$, or that $x$ \emph{does not appear} in $r$. In what follows, we will need Gr\"obner bases. We refer the reader to \cite{Green,GHS} for the notions needed concerning Gr\"obner bases. We fix an \emph{admissible order} on the set of all paths of $Q$ and all Gr\"obner bases will be with respect to that given order. The \emph{tip} of $x \in kQ$ is the largest path occurring in $x$, with respect to that order.

\begin{Lemma}  Let $\Lambda = kQ/I$ and assume that $x \in Q_0 \cup Q_1$ is such that the $r_i$ do not involve $x$ for all $1 \le i \le t$. Then there is a Gr\"obner basis of the ideal $I$ consisting of uniform elements of $kQ$, all of which do not involve $x$.
\end{Lemma}

\begin{proof}
We start with the generators $\{r_1, \ldots, r_t\}$ of $I$ and we apply Corollary 2.11 of \cite{GHS} to get a finite set $\{s_1, \ldots, s_m\}$ of uniform tip-reduced elements of $kQ$ that generate $I$. This new set of generators will still have the same property that the $s_i$ do not involve $x$, since they are obtained from the $r_i$ by applying simple reductions. Recall that an \emph{overlap relation} between $x, y \in kQ$ is an element $\lambda_x xm - \lambda_x ny$ where $m,n$ are paths of length at least one, the length of $m$ is strictly less than that of tip$(y)$,  ${\rm tip}(x)m = n{\rm tip}(y)$ and $\lambda_x, \lambda_y$ are the non-zero coefficients of ${\rm tip}(x), {\rm tip}(y)$ in $x,y$, respectively. We first observe that any overlap relation between $s_i, s_j$ is such that the paths $m,n$ do not involve $x$.
If such a relation cannot be reduced to zero, it can be reduced to a linear combination of paths still not involving $x$. Now, according to Theorem 2.13 in \cite{GHS} and the remark before it (see also Section 2.4.1 in \cite{Green}), we see that after a possibly infinite number of steps, we can get a (possibly infinite) Gr\"obner basis of $I$ consisting only of linear combinations of paths not involving $x$.
\end{proof}

\begin{Cor}
Assume that $x \in Q_0 \cup Q_1$ is such that the $r_i$ do not involve $x$. Then the algebra $\Lambda = kQ/I$ has a basis $\mathcal{N}$ consisting of residue classes of paths such that if $p_1, p_2$ lie in that basis, then so does $p = p_1 x p_2$.
\end{Cor}

\begin{proof}
By Lemma 2.6 in \cite{GHS}, a basis $\mathcal{N}$ of $\Lambda$ comes from the residue class of paths $t$ such that no subpath of $t$ is the tip of an element in the constructed Gr\"obner basis. Assume that both $p_1, p_2$ lie in that basis $\mathcal{N}$. Since the tip of any element in our Gr\"obner basis is a path not involving $x$, we see that $p$ has no subpath equal to the tip of an element in our Gr\"obner basis.
\end{proof}

\begin{Prop}
Let $\Lambda = kQ/I$ be such that $I$ is generated by relations not involving a given element $x$ of $Q_0 \cup Q_1$. Then the two-sided ideal generated by $x + I$ in $\Lambda$ is projective as a left and as a right $\Lambda$-module.
\end{Prop}

\begin{proof}
We only prove the left version where $x = \alpha$ is an arrow. Assume $\alpha: s \to t$ and let $\mathcal{N}$ be the basis from the above corollary. Let $p_1, \ldots, p_r$ be the paths in $\mathcal{N}$ ending in $s$ and let $q_1, \ldots, q_s$ be the paths in $\mathcal{N}$ starting in $t$. The $p_i$ form a basis of $e_s\Lambda$ and the $q_j$ form a basis of $\Lambda e_t$. By the corollary, we see that the $q_j\alpha p_i$ form a basis of $\Lambda \alpha \Lambda$. Let $\pi: (\Lambda e_t)^r \to \Lambda \alpha \Lambda$ be the morphism given by $\pi(a_1, \ldots, a_r) = \sum a_i\alpha p_i$. It is clear that $\pi$ is surjective. Now, the remark above implies it is injective. This proves that $\Lambda \alpha \Lambda$ is a projective left $\Lambda$-module.
\end{proof}

Now, assume that $\alpha: s \to t$ is an arrow of $Q$. Let $Q'$ be the quiver obtained from $Q$ by removing the arrow $\alpha$ and replacing it with a path of length two $\alpha_2 \alpha_1$ where $\alpha_1: s \to u$ and $\alpha_2: u \to t$, and where $u$ is a new vertex of $Q'$. We let $I$ be the same ideal, but seen as an ideal of $kQ'$. We let $B = kQ'/I$. Observe that $B/Be_uB \cong \Lambda/\Lambda \alpha \Lambda$. Moreover, it is clear that $Be_uB$ is a projective left $B$-module, since the $r_i$ do not involve $u$. It follows from Proposition \ref{AHLU} that findim$\,B < \infty$ if and only if findim$\,B/Be_uB< \infty$. Therefore, we get the following, which has first been proven in \cite{GSP}, in case $x$ is an arrow.

\begin{Prop}
Let $\Lambda = kQ/I$ be such that $I$ is generated by relations not involving a given element $x$ of $Q_0 \cup Q_1$. Let $J$ be the two sided ideal generated by $x$. Then
\begin{enumerate}[$(1)$]
    \item ${\rm findim}\Lambda \le 2\,{\rm findim}\Lambda/J +3$.
\item ${\rm findim}\Lambda/J \le {\rm findim}\Lambda.$
\item ${\rm findim}\Lambda < \infty$ if and only if ${\rm findim}\Lambda/J < \infty$.\end{enumerate}
\end{Prop}

\begin{proof}
Assume that ${\rm findim}\,\Lambda/J=r < \infty$. Then ${\rm findim}\,B/Be_uB=r$, so ${\rm findim}\,B \le  r+2$ by Proposition \ref{AHLU}. Observe that $\Lambda = (1-e_u)B(1-e_u)$ and that the simple $B$-module at $u$ has projective dimension $1$. Therefore, by Theorem \ref{maintheo}, we have ${\rm findim}\,\Lambda \le 2(r+2) - \ell \le 2(r+2)-1$ since the graded loewy length $\ell$ of $Y(e)$ is at least one. This proves (1). Statement (2) follows from the fact that $J$ is projective; see Proposition \ref{AHLU}, part (2).
\end{proof}

\subsection{Almost vanishing ideals}

In this subsection, we consider a left ideal $J$ such that $J = Je$ for some primitive idempotent $e$. We assume further that $J {\rm rad}\Lambda = 0$. This implies that $J$ is actually a two-sided ideal. If $J$ is not included in the radical of $\Lambda$, then $J = \Lambda e$ with $J {\rm rad}\Lambda = 0$ so $e$ has to be a source vertex. It is easy to see that ${\rm findim}\,\Lambda < \infty$ if and only if ${\rm findim}\,\Lambda/J < \infty$ in that case. Therefore, we will assume that $J \subseteq {\rm rad} \Lambda$. Hence, $J$ is an $\Lambda/J$-module. Now, there are uniform elements $\{r_1, \ldots, r_t\}$ of $kQ$ with $r_i = r_ie$ for $1 \le i \le t$ that generate $J$ and such that $r_i {\rm rad}\Lambda = 0$ for all $i$. We may assume that $S := \{r_1, \ldots, r_t\}$ is a minimal generating set of $J$.

We construct a new quiver $Q_J$ by adding a new vertex $x$ and the following arrows to $Q$.
We add an arrow $\alpha: v \to x$. For each $r_i \in S$, we add an arrow $\beta_{i} : x \to t(r_i)$, where $t(r_i)$ is the terminal vertex of $r_i$ (recall $r_i$ is uniform). We define an ideal $I_J$ of $Q_J$ by adding relations to $I \subseteq kQ_J$ as follows. For each $r_i \in S$, we add $r_i - \beta_{i}\alpha$; for each $\gamma \in (Q_J)_1$ with $t(\gamma)=v$, we add $\alpha\gamma$. Finally, we impose that $\sum_{i=1}^{t}a_i\beta_{i} \in I_J$ if and only if $\sum_{i=1}^{t}a_ir_i \in I$, where the $a_i$ are in $kQ \subset kQ_J$. We set $B = kQ_J/I_J$.

\begin{Lemma} \label{lemma3.6}
With the notations of the previous paragraph, $Be_xB$ is a projective left ideal of $B$ with $B/Be_xB \cong \Lambda/J$.
\end{Lemma}

\begin{proof}
The fact that $B/Be_xB \cong \Lambda/J$ follows from the definition of $B$.
Observe that $Be_{x}B = Be_xB(e + e_{x}) = Be_xBe \oplus Be_x$, where the second summand is projective.
Therefore, it remains to show that $Be_xBe=B\alpha \cong  Be_{x}$, which amounts to proving that if $b = be_{x}$ is non-zero in $B$, then $b\alpha$ is non-zero in $B$. Therefore, assume that $b = be_{x}$ is such that $b\alpha \in I_J$. In particular, $b$ is represented by a linear combination of paths of positive lengths in $kQ_J$. Therefore, $b = b_1\beta_{1} + \cdots + b_{t}\beta_{t} + I_J$ and where the $b_i$ are in $kQ$. Now, $b\alpha = b_1r_1 + \cdots b_{t}r_{t}$ lies in $I$, which is equivalent to $b_1\beta_1 + \cdots b_{t}\beta_{t}$ lying in $I_J$, a contradiction.
\end{proof}

\begin{Prop} \label{lastprop}
Let $J$ be a submodule of $\Lambda e$ for $e$ primitive and assume that $J {\rm rad} \Lambda=0$. Thus, $J$ is a two-sided ideal.
Assume further that ${\rm pd}_{\Lambda/J}J < \infty$. Then ${\rm findim}\,\Lambda \le 2\,{\rm findim}\,\Lambda/J + 3 $.
\end{Prop}

\begin{proof}
Assume that ${\rm findim}\,\Lambda/J = r < \infty$. Consider the algebra $B$ as constructed above, with idempotent $e_x$ such that $\Lambda/J \cong B/Be_xB$. We first observe that ${\rm findim}\,B \le r + 2$. This follows from Proposition \ref{AHLU} and Lemma \ref{lemma3.6}.
Consider the simple module $S_{x}$ supported at $e_x$ in $B$. Observe that its first syzygy is $\Omega =\sum_{i}B\beta_i$. Note that $e_x\Omega=0$, so ${\rm pd}_B \Omega < \infty$ if and only if ${\rm pd}_{B/Be_xB} \Omega < \infty$. However, through the isomorphism $B/Be_xB \cong \Lambda/J$, $\Omega$ corresponds to $J$. Therefore, by using the hypothesis, this yields ${\rm pd}_B \Omega < \infty$, so $S_x$ has finite projective dimension. By Theorem \ref{maintheo}, this yields ${\rm findim}\,\Lambda \le 2(r+2)- \ell  \le 2r+3$.
\end{proof}

We end with an example to illustrate this result.

\begin{Exam} Let $Q$ be the quiver given by
$$\xymatrixrowsep{10pt}\xymatrixcolsep{10pt}
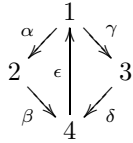
\xymatrix{
 & 1 \ar[dl]_\alpha \ar[dr]^\gamma & \\
2 \ar[dr]_\beta && 3 \ar[dl]^\delta\\
& 4 \ar[uu]^\epsilon&
}$$
with admissible ideal $I=\langle \beta\alpha - \delta\gamma, \epsilon\delta, \gamma\epsilon, \alpha\epsilon\beta \rangle$. Let $\Lambda = kQ/I$. Consider the two-sided ideal $J = \langle \alpha \epsilon \rangle$. It clearly satisfies the first hypothesis of Proposition \ref{lastprop}, since $J$ is a one-dimensional radical ideal with $J = Je_4$. Observe that as an $\Lambda/J$-module, $J$ has projective resolution
$$0 \to \frac{\Lambda e_4}{Je_4} \to \frac{\Lambda e_2}{Je_2} \to J \to 0$$
where $Je_4 = J$ and $Je_2 = 0$. Therefore, Proposition \ref{lastprop} yields that ${\rm findim}\,\Lambda \le 2\,{\rm findim}\,\Lambda/J + 3$. It is easily checked that the global dimension of $\Lambda/J$ is $4$. Therefore, ${\rm findim}\,\Lambda \le 11$. In fact, $\Lambda$ is of finite representation type and has exactly $8$ indecomposable modules of finite projective dimension. They are the $4$ indecomposable projective modules, the simple modules $S_3, S_4$, the quotient $\Lambda e_1/{\rm soc}(\Lambda e_1)$, and the two-dimensional module $M$ which is the cokernel of the inclusion $\Lambda e_3 \to \Lambda e_1$. The finitistic dimension of $\Lambda$ is $3$.
\end{Exam}

One easy consequence of Proposition \ref{lastprop} is the following.

\begin{Cor}
Let $S$ be a simple submodule of an indecomposable projective $\Lambda$-module of maximal loewy length. Then $S$ is a two-sided ideal of $\Lambda$. Assume that ${\rm pd}_{\Lambda/S}S$ is finite. Then ${\rm findim}\,\Lambda \le 2\,{\rm findim}\,\Lambda/S + 3$.
\end{Cor}

\end{document}